\documentclass{amsart}
\usepackage{amsmath,amsthm,amssymb}
\usepackage{graphicx}

\newtheorem{theorem}{Theorem}[section]
\newtheorem{proposition}[theorem]{Proposition}

\newtheorem{corollary}[theorem]{Corollary}

\theoremstyle{definition}
\newtheorem{definition}[theorem]{Definition}
\newtheorem{problem}[theorem]{Problem}
\newtheorem{example}[theorem]{Example}

\theoremstyle{remark}
\newtheorem{remark}[theorem]{Remark}

\theoremstyle{theorem}

\title[Nonexistence of twists and surgeries generating exotic 4-manifolds]{Nonexistence of twists and surgeries generating exotic 4-manifolds}
 
\author[Kouichi Yasui]{Kouichi Yasui}
\date{December 1, 2016. \textit{Revised}: August 29, 2018}

\address{Department of Pure and Applied Mathematics, Graduate School of Information Science and Technology, Osaka University, 
1-5 Yamadaoka, Suita, Osaka 565-0871, Japan}
\email{kyasui@ist.osaka-u.ac.jp}
\begin{document}

\begin{abstract} 
It is well known that for any exotic pair of simply connected closed oriented 4-manifolds, one is obtained from the other by twisting a compact contractible submanifold via an involution on the boundary. By contrast, here we show that for each positive integer $n$, there exists a simply connected closed oriented 4-manifold $X$ such that, for any compact (not necessarily connected) codimension zero submanifold $W$ with $b_1(\partial W)<n$, the set of all smooth structures on $X$ cannot be generated from $X$ by twisting $W$ and varying the gluing map. 
As a corollary, we show that there exists no `universal' compact 4-manifold $W$ such that, for any simply connected closed 4-manifold $X$, the set of all smooth structures on $X$ is generated from a 4-manifold by twisting a fixed embedded copy of $W$ and varying the gluing map. Moreover, we give similar results for surgeries. 
\end{abstract}

\maketitle

\section{Introduction}\label{sec:intro}It has been a fundamental problem in 4-dimensional topology to construct all exotic (i.e.\ homeomorphic but not diffeomorphic) smooth structures on a 4-manifold, but there is still no known 4-manifold whose all exotic smooth structures are found. A powerful method for constructing exotic smooth structures is twisting  a compact codimension zero submanifold in a 4-manifold, that is, removing the submanifold and regluing it differently. 
In fact, a well-known theorem states that for any exotic pair of simply connected closed smooth oriented 4-manifolds, one is obtained from the other by twisting a compact contractible codimension zero submanifold via an involution on the boundary (\cite{CFHS}, \cite{M}). A pair of such a contractible 4-manifold and an involution on its boundary is called a cork (\cite{K}, \cite{AY1}), and due to the order of the gluing map, it is sometimes called a cork of order 2 (\cite{T15}). Corks have been actively studied from various point of view, and have many interesting applications (e.g.\ \cite{A2}, \cite{AY3}, \cite{AY5}, \cite{AKMR}, \cite{AR}, \cite{Y7}).  However, since there are infinitely many corks, and many contractible 4-manifolds admit infinitely many embeddings, it is still not known how to produce all exotic smooth structures. 

Recently higher order corks (\cite{T0}, \cite{AKMR2}) and surprisingly infinite order corks (\cite{G16a}, see also \cite{A16}, \cite{G16b}, \cite{T2}) were discovered, but interestingly Tange~\cite{T1} showed that a natural extension of the above cork theorem does not always hold for an infinite exotic family, by showing a certain finiteness for Ozsv\'{a}th-Szab\'{o} invariants of cork twisted 4-manifolds.  
More precisely, he gave infinite families of pairwise exotic closed 4-manifolds such that, for any  4-manifold $X$, any contractible submanifold $W$, and any self-diffeomorphism $f$ of $\partial W$, the families cannot be constructed from $X$ by twisting $W$ via powers of $f$. 

In this paper we discuss generalizations of the cork theorem from viewpoints of twists and more general surgeries. 

\subsection{Nonexistence of twists generating all exotic smooth structures}
We first discuss twists, using the following terminologies. 

\begin{definition}\label{intro:def1}A pair $(X,W)$ consisting of a smooth oriented 4-manifold $X$ and a compact (not necessarily connected) codimension zero submanifold $W$ will be called a \textit{$4$-manifold pair}. 
For a family of smooth oriented 4-manifolds, we say that the family is \textit{generated from $X$ by twisting $W$}, if each member is orientation preserving diffeomorphic to a 4-manifold obtained from $X$ by removing the submanifold $W$ and gluing it back via a (not necessarily orientation preserving) self-diffeomorphism of the boundary $\partial W$. Here the orientations of the resulting 4-manifolds are those induced from the complement of $W$ in $X$. Hence, in the case where the gluing map reverses the orientation, the newly glued piece is the orientation reversal $\overline{W}$ of $W$. 
\end{definition}

For example, it is well known that logarithmic transforms (i.e.\ twists along $T^2\times D^2$) can generate various interesting infinite families of exotic 4-manifolds under a certain condition (cf.\ \cite{GS}, \cite{FS4}, \cite{A_book}). 
For an oriented smooth 4-manifold $X$, let $\mathcal{S}(X)$ be the set of all smooth structures on $X$, that is, $\mathcal{S}(X)$ is the set of all oriented smooth 4-manifolds homeomorphic to $X$ preserving the orientations, considered up to orientation preserving diffeomorphisms. This set was inspired from Tange's \textit{galaxy} (\cite{T2}). As is well-known, $\mathcal{S}(X)$ is a countable set for any compact oriented 4-manifold $X$, since members have at most countably many choices of handle decompositions. We consider the following natural problem. 

\begin{problem}\label{intro:prob:twists}Does a given compact  oriented smooth 4-manifold $X$ admit a compact (not necessarily connected) codimension zero submanifold $W$ such that $\mathcal{S}(X)$ is generated from $X$ by twisting $W$?
\end{problem}

This problem asks a generalization of the cork theorem, since we do not impose any restrictions on the topology of $W$ and on the gluing map. We note that a self-diffeomorphism of $\partial W$ may not preserve the connected components of $\partial W$ when $\partial W$ is disconnected, and hence twists in this problem are more flexible when $W$ is disconnected. If the answer is affirmative, then we have a very useful approach for finding and understanding all exotic smooth structures on a 4-manifold, since they are generated by twisting just a single submanifold in this case. 
However, we give a partial negative answer under a mild assumption on $b_1(\partial W)$.

\begin{theorem}\label{intro:thm:twists}For each positive integer $n$, there exists a compact connected oriented smooth 4-manifold $X$ such that, for any compact $($not necessarily connected$)$ codimension zero submanifold $W$ satisfying $b_1(\partial W)<n$, the set $\mathcal{S}(X)$ cannot be generated from $X$ by twisting $W$. Furthermore, there exist infinitely many pairwise non-homeomorphic such simply connected closed oriented smooth 4-manifolds. 
\end{theorem}
Note that $W$ and $\partial W$ are not necessarily connected. This result makes a sharp contrast with the aforementioned cork theorem. We note that the aforementioned Tange's result follows from  this result, since the boundary of any compact contractible 4-manifold is a (connected) homology 3-sphere and thus satisfies $b_1=0$. Our proof is completely different from Tange's one and gives a topologically intuitive understanding. 

To prove this theorem, we introduce the notion of an \textit{adjunction $n$-genus}, which is a diffeomorphism invariant of 4-manifolds determined by the minimal genus function. 
We moreover give a sufficient condition that an infinite family of pairwise exotic (not necessarily closed) 4-manifolds cannot be generated by twisting, by using the adjunction $n$-genera. This sufficient condition gives various examples of the above theorem including many well-known 4-manifolds, and we here list a few simple ones in the $n=1$ case. 

\begin{example}\label{intro:ex}$(1)$ \textit{Elliptic surfaces}. Every elliptic surface $E(k)$ $(k>1)$ without multiple fibers is an example of the $n=1$ case of Theorem~\ref{intro:thm:twists}. In fact we will show that, for any infinite family of homeomorphic but pairwise non-diffeomorphic simply connected closed elliptic surfaces with $b_2^+>1$, the family does not admit a 4-manifold pair $(X,W)$ with $b_1(\partial W)<1$ such that the family is generated from $X$ by twisting $W$. \\
$(2)$ \textit{Knot surgery}. Suppose that a simply connected closed oriented smooth 4-manifold $X$ with $b_2^+>1$ and a non-trivial Seiberg-Witten invariant contains a homologically non-trivial torus $T$ of self-intersection $0$ satisfying $\pi_1(X-T)\cong 1$. Then $X$ is an example of the $n=1$ case of Theorem~\ref{intro:thm:twists}. Indeed, for any infinite family $\{K_i\}$ of knots in $S^3$ whose Alexander polynomials have pairwise distinct degrees, we show that the infinite family $\{X_{K_i}\}$ of homeomorphic but pairwise non-diffeomorphic simply connected closed 4-manifolds does not admit a 4-manifold pair $(Z,W)$ with $b_1(\partial W)<1$ such that the family is generated from $Z$ by twisting $W$. Here each $X_{K_i}$ denotes the 4-manifold obtained from $X$ by performing Fintushel-Stern's knot surgery~\cite{FS2} on $T$ using $K_i$. \\
$(3)$ \textit{Small Stein 4-manifolds}. We can also provide non-closed examples. For example, we consider infinite families of homeomorphic but pairwise non-diffeomorphic simply connected compact Stein 4-manifolds with $b_2=2$ obtained by Akbulut and the author \cite{AY6} and the author \cite{Y6}. We show that any family does not admit a 4-manifold pair $(X,W)$ with $b_1(\partial W)<1$ such that the family is generated from $X$ by twisting $W$. 
\end{example}

Due to the cork theorem and the existence of infinite order corks, one might hope for the existence of a `universal' 4-manifold $W$ such that, for any compact 4-manifold $X$, the set of all smooth structures on $X$ is generated from a 4-manifold by twisting a fixed embedded copy of $W$. However, the above theorem dashes this hope.  

\begin{corollary}\label{intro:cor:twist:universal}There exists no compact $($not necessarily connected$)$ oriented smooth 4-manifold $W$ such that for any simply connected oriented closed smooth 4-manifold $X$, the set $\mathcal{S}(X)$ is generated from a smooth oriented 4-manifold by twisting a fixed embedded copy of $W$. 
\end{corollary}

\subsection{Nonexistence of surgeries generating all exotic smooth structures}
Next we discuss surgeries using the following terminology. 

\begin{definition} For a 4-manifold pair $(X,W)$ and a family of smooth oriented 4-manifolds, 
we say that the family is \textit{generated from $X$ by performing surgeries on $W$}, if each member is obtained from $X$ by removing the submanifold $W$ and gluing a compact oriented smooth 4-manifold whose boundary is orientation preserving diffeomorphic to $\partial W$. Note that we do not fix the newly glued piece.
\end{definition}

Clearly surgeries can produce many more 4-manifolds than twists. For example, one can replace submanifolds with their exotic copies. Since surgeries can produce various infinite exotic families (cf.\ \cite{GS}, \cite{FS4}, \cite{A_book}), it would be natural to consider a surgery version of Problem~\ref{intro:prob:twists}. 

\begin{problem}\label{intro:prob:surgery}
Does a given compact  oriented smooth 4-manifold $X$ with $b_2>0$ admit a compact $($not necessarily connected$)$ codimension zero submanifold $W$ with $b_2(W)<b_2(X)$ such that $\mathcal{S}(X)$ is generated from $X$ by performing surgeries on $W$? 
\end{problem}

The conditions $b_2(X)>0$ and $b_2(W)<b_2(X)$ might look unnatural, but we impose this restriction to avoid a trivial affirmative answer. Indeed, for any compact codimension zero submanifold $V$ of the 4-ball, it is easy to see that $W=X-\text{int}\, V$ provides an affirmative answer, realizing any integer not less than $b_2(X)$ as $b_2(W)$. 

Here we show that this problem also has a partial negative answer under a mild assumption on $b_2(W)+3b_1(\partial W)$, even though surgeries are much more flexible than twists.  

\begin{theorem}\label{intro:thm:surgery}For each positive integer $n$, there exists a compact connected oriented smooth 4-manifold $X$ with $b_2>n$ such that, for any compact $($not necessarily connected$)$ codimension zero submanifold $W$ with $b_2(W)+3b_1(\partial W)<n$, the set $\mathcal{S}(X)$ cannot be generated from $X$ by performing surgeries on $W$. Furthermore, there exist infinitely many pairwise non-homeomorphic such simply connected closed oriented smooth 4-manifolds. 
\end{theorem}


It would be worth to remark that the above $X$ can be chosen so that it satisfies the condition of Theorem~\ref{intro:thm:twists}. We also give a sufficient condition that an infinite family of pairwise exotic (not necessarily closed) 4-manifolds cannot be generated from a 4-manifold by performing surgeries, by using the adjunction $n$-genera. Similarly to the case of twists, this sufficient condition gives various examples including many well-known 4-manifolds. For example, the 4-manifolds in Example~\ref{intro:ex} are examples of the $n=1$ case of this theorem as well. 

Similarly to the case of twists, Theorem~\ref{intro:thm:surgery} shows the nonexistence of a `universal' 4-manifold even for surgeries. 

\begin{corollary}\label{intro:cor:surgery:universal}There exists no compact oriented smooth 4-manifold $W$ such that for any simply connected closed oriented smooth 4-manifold $X$, the set $\mathcal{S}(X)$ is generated from $X$ by performing surgeries on a fixed embedded copy of $W$.  
\end{corollary}

\subsection{Nonexistence of twists generating all exotic smooth structures by varying embeddings}
Finally, we discuss a further generalization of the cork theorem.  

\begin{problem}\label{intro:prob:embedding}Does a given compact oriented smooth 4-manifold $X$ admit a compact $($not necessarily connected$)$ oriented smooth 4-manifold $W$ such that $\mathcal{S}(X)$ is generated from $X$ by twisting an embedded copy of $W$ and varying the embedding of $W$ into $X$?
\end{problem}

The construction in this problem is much more flexible than the one in Problem~\ref{intro:prob:twists}, since we vary an embedding of $W$. A related problem was studied by Akbulut and the author in \cite{AY3} and \cite{AY4}, and it was shown that many order-2 corks can produce infinite families of pairwise exotic simply connected closed 4-manifolds by twisting the corks via the involutions and varying their embeddings, though they can produce only exotic pairs in the case where their embeddings are fixed. By contrast, we give a partial negative answer to this problem.
  
\begin{theorem}\label{intro:thm:embedding:closed}For each positive integer $n$, there exists a simply connected closed oriented smooth 4-manifold $X$ with $b_2=12n+10$ such that for any compact oriented smooth 4-manifold $W$ with $b_2(W)-4b_1(\partial W)>11n+10$, the set $\mathcal{S}(X)$ cannot be generated from $X$ by twisting an embedded copy of $W$ and varying the embedding of $W$ into $X$. 
\end{theorem}

Actually it is possible to decrease the lower bound for $b_2(W)-4b_1(\partial W)$ by improving the construction of $X$, but we do not pursue this point in the case where $X$ is closed. In the case where $X$ is a 4-manifold having boundary, we can significantly decrease the lower bound. 

\begin{theorem}\label{intro:thm:embedding:boundary}For each positive integer $n$, there exists a compact connected oriented smooth 4-manifold $X$ with $b_2=2n$ such that for any compact oriented smooth 4-manifold $W$ with $b_2(W)-4b_1(\partial W)>n$, the set $\mathcal{S}(X)$ cannot be generated from $X$ by twisting an embedded copy of $W$ and varying the embedding of $W$ into $X$. 
\end{theorem}

As before, by using the adjunction $n$-genera, we also give a sufficient condition that an infinite family of pairwise exotic (not necessarily closed) 4-manifolds cannot be generated from a 4-manifold by twisting an embedded copy of a 4-manifold and varying the embedding, 

To prove our main results, we give sufficient conditions that infinitely many members of a family of 4-manifolds have pairwise distinct adjunction $n$-genera, by applying the adjunction inequalities. As a corollary, we obtain a simple but effective method for `coarsely' distinguishing smooth structures of 4-manifolds admitting Stein structures. 


\section{Adjunction $n$-genus}\label{sec:genus}
In this section, we introduce the notion of an adjunction $n$-genus together with our conventions. We also briefly recall genus invariants of 4-manifolds. 

For a smooth 4-manifold $X$ and a second homology class $\alpha$, the minimal genus $g_X(\alpha)$ of $\alpha$ is defined to be the minimal genus of a smoothly embedded closed oriented surface representing $\alpha$ (cf.\ \cite{GS}).  Here note that the genus of a disconnected closed oriented surface is defined to be the sum of the genera of all connected components. If $X$ is connected, then for any disconnected closed oriented surface in $X$, one can construct a connected closed oriented surface of the same genus representing the same homology class. 

In general, it is  difficult to determine the minimal genus of a second homology class. 
A strong method for determining the minimal genus is the inequalities below, which are called the adjunction inequalities. In the rest, the self-intersection number of a second homology class $\alpha$ is denoted by $\alpha\cdot \alpha$. 

\begin{theorem}[\cite{KM1}, \cite{MST}, \cite{OzSz_ad}]
Assume that $X$ is a closed connected oriented smooth $4$-manifold with $b_2^+>1$. Let $K$ be a Seiberg-Witten basic class of $X$, and let $\alpha$ be a non-zero class of $H_2(X;\mathbb{Z})$. If $\alpha\cdot \alpha\geq 0$, then 
\begin{equation*}
\lvert \langle K, \alpha\rangle \rvert + \alpha\cdot \alpha \leq 2g_X(\alpha)-2. 
\end{equation*}
If $\alpha\cdot \alpha<0$, and $X$ is of Seiberg-Witten simple type, then 
\begin{equation*}
\lvert \langle K, \alpha\rangle \rvert + \alpha\cdot \alpha \leq \max\{2g_X(\alpha)-2,0\}. 
\end{equation*}
\end{theorem}

The function $g_X: H_2(X;\mathbb{Z})\to \mathbb{Z}$ is called the \textit{minimal genus function} of $X$. This function has useful information about smooth structures of 4-manifolds, but in general, it is significantly more difficult to distinguish the functions of two 4-manifolds, since we need to consider all possible identifications of their second homology groups. The author~\cite{Y5} (cf.\ \cite{Y6}) thus introduced the \textit{relative genus function}, which is a genus invariant (i.e.\ a diffeomorphism invariant determined by the minimal genus function on a 4-manifold). Gompf~\cite{G_GT} subsequently introduced a different genus invariant called the \textit{genus rank function}, and the author~\cite{Y8} recently introduced a new genus invariant called \textit{intersection genus}. In this paper we introduce yet another genus invariant to prove our main results.

Here we introduce our conventions. Let $X$ be an oriented smooth 4-manifold such that $H_2(X;\mathbb{Z})$ is finitely generated. If $X$ is compact, then this condition clearly holds. 
A finite subset $\mathbf{v}=\{v_1,v_2,\dots,v_m\}$ of $H_2(X;\mathbb{Z})$ will be called a \textit{rational basis} of $H_2(X;\mathbb{Z})$, if $\mathbf{v}$ becomes a basis of $H_2(X;\mathbb{Q})$ through the natural homomorphism. In the case where $b_2(X)=0$, we regard the empty set as the rational basis of $H_2(X;\mathbb{Z})$.

Now we introduce the notion of an adjunction $n$-genus.

\begin{definition}(1) The function $g_X^{ad}:H_2(X;\mathbb{Z})\to \mathbb{Z}$ defined by 
\begin{equation*}
g_X^{ad}(v)=2g_X(v)-v\cdot v
\end{equation*}
 will be called the \textit{adjunction genus function} of $X$, and the value of $g_X^{ad}(v)$ will be called the \textit{adjunction genus} of $v$. \\
(2) For a rational basis $\mathbf{v}=\{v_1,v_2,\dots,v_m\}$ of $H_2(X;\mathbb{Z})$ and a positive integer $n\leq b_2(X)$, we define a non-negative integer $G_{X,\mathbf{v},n}$ as the $n$-th largest value of $\{g_X^{ad}(v_1), g_X^{ad}(v_2), \dots, g_X^{ad}(v_m)\}$, where the order is counted with multiplicity. For a positive integer $n>b_2(X)$, we put $G_{X,\mathbf{v},n}=0$. 
 For a non-negative integer $n$, we define an integer $G_{X}(n)$ by 
\begin{equation*}
G_{X}(n)=\left\{
\begin{array}{ll}
\min \{G_{X,\mathbf{v},n}\mid \text{$\mathbf{v}$ is a rational basis of $H_2(X;\mathbb{Z})$}\},&\text{if $n\geq 1$;} \\
0,&\text{if $n=0$.}
\end{array}
\right. 
\end{equation*}We call $G_{X}(n)$ the \textit{adjunction $n$-genus} of $X$. 
\end{definition}

The adjunction $n$-genus is clearly a diffeomorphism invariant of oriented 4-manifolds for each fixed integer $n$. We can define a similar invariant using the minimal genus function, but the adjunction $n$-genus is easier to estimate due to the adjunction inequality. One can also define an integral version of the adjunction $n$-genus by using (ordinary) bases instead of rational bases, but we do not pursue this point here. In general, computing adjunction $n$-genera is a very hard problem, but estimating their values are useful for studying smooth properties of (families of) 4-manifolds.


\begin{remark}\label{rem:n-1}For each $n\geq 2$, the inequality $G_{X}(n-1)\geq G_{X}(n)$ holds. Thus for any infinite family of 4-manifolds with pairwise distinct adjunction $n$-genera, at least infinitely many members have pairwise distinct adjunction $(n-1)$-genera. 
\end{remark}

\section{Adjunction $n$-genera, twists and surgeries}\label{sec:nonexistence}

In this section, utilizing adjunction $n$-genera, we give sufficient conditions that an infinite family of 4-manifolds cannot be generated by twists and/or surgeries. We begin with twists.

\begin{theorem}\label{sec:twist:thm:general}
Suppose that infinitely many members of a family of compact connected oriented smooth 4-manifolds have pairwise distinct adjunction $n$-genera for a positive integer $n$. Then for any 4-manifold pair $(X,W)$ with $b_1(\partial W)<n$, the family cannot be generated from $X$ by twisting $W$. $($Note that $W$ and $\partial W$ are not necessarily connected.$)$ 
\end{theorem}


To prove this theorem, we establish the following finiteness result.

\begin{proposition}\label{sec:twist:prop:general}
Suppose that a family of compact connected oriented smooth 4-manifolds is generated from $X$ by twisting $W$ for a 4-manifold pair $(X,W)$. Then, for any integer $n>b_1(\partial W)$, there are at most finitely many integers that can be the values of the adjunction $n$-genera of the members. 
\end{proposition}

We first prove the $b_1(\partial W)=0$ case, since its proof is short and nicely demonstrates our main idea. 

\begin{proof}[Proof of the $b_1(\partial W)=0$ case of Proposition~\ref{sec:twist:prop:general}]Let $\{X_\lambda\}_{\lambda\in \Lambda}$ be a family of 4-manifolds satisfying the assumption. We first consider the case where $W$ is connected. Then each $X_\lambda$ is decomposed as either 
\begin{equation*}
X_\lambda=(X-\textnormal{int}\,W)\cup W \quad \textnormal{or} \quad X_\lambda=(X-\textnormal{int}\,W)\cup \overline{W}.
\end{equation*}
 We fix rational bases $\mathbf{v}=\{v_1,v_2,\dots,v_s\}$ and $\mathbf{w}=\{w_1,w_2,\dots,w_t\}$ of $H_2(X-\textnormal{int}\,W;\mathbb{Z})$ and $H_2(W;\mathbb{Z})$ respectively. We note that $\mathbf{w}$ is a rational basis of $H_2(\overline{W};\mathbb{Z})$, since this group is equal to $H_2(W;\mathbb{Z})$. We regard each $v_i$ (resp.\ $w_j$) as an element of $H_2(X_\lambda;\mathbb{Z})$ via the inclusion $X-\textnormal{int}\,W\hookrightarrow X_\lambda$ (resp.\ either the inclusions $W\hookrightarrow X_\lambda$ or $\overline{W}\hookrightarrow X_\lambda$). Due to the assumption $b_1(\partial W)=0$, the Mayer-Vietoris exact sequence for the above decomposition of each $X_\lambda$ implies that $\mathbf{v}\cup \mathbf{w}$ is a rational basis of $H_2(X_\lambda;\mathbb{Z})$. Since each $v_i$ (resp.\ $w_j$) is represented by a surface in $X-\textnormal{int}\,W$ (resp.\ either $W$ or $\overline{W}$), the definition of the adjunction $n$-genus shows that 
\begin{equation*}
0\: \leq \: G_{X_\lambda}(n) \: \leq \: G_{X_\lambda,\mathbf{v}\cup \mathbf{w},n} \: \leq \: \max\{G_{X-\textnormal{int}\,W, \mathbf{v},1}, \;G_{W,\mathbf{w},1}, \;G_{\overline{W},\mathbf{w},1}\}. 
\end{equation*}
Note that the rightmost value is independent of $\lambda$.  Hence at most finitely many integers can be the values of the adjunction $n$-genera of the members. 

In the case where $W$ is disconnected, each $X_\lambda$ is decomposed as $(X-\textnormal{int}\,W)\cup \widehat{W}$, where $\widehat{W}$ is a copy of $W$ possibly equipped with the reverse orientations for some of its connected components, depending on $\lambda$. Since $W$ has at most finitely many connected components due to the compactness condition of $W$, there are at most finitely many choices of $\widehat{W}$. Therefore, we can prove the desired finiteness similarly to the connected case. 
\end{proof}

Now we prove the general case. 

\begin{proof}[Proof of Proposition~\ref{sec:twist:prop:general}]Let $\{X_\lambda\}_{\lambda\in \Lambda}$ be a family of 4-manifolds satisfying the assumption. We first consider the case where $W$ is connected. Then each $X_\lambda$ is decomposed as either $X_\lambda=(X-\textnormal{int}\,W)\cup W$ or $X_\lambda=(X-\textnormal{int}\,W)\cup \overline{W}$. As before, we fix rational bases $\mathbf{v}=\{v_1,v_2,\dots,v_s\}$ and $\mathbf{w}=\{w_1,w_2,\dots,w_t\}$ of $H_2(X-\textnormal{int}\,W;\mathbb{Z})$ and $H_2(W;\mathbb{Z})$ respectively, and we regard each $v_i$ and $w_j$ as elements of $H_2(X_\lambda;\mathbb{Z})$ via the inclusions. The Mayer-Vietoris exact sequence for the decomposition of $X_\lambda$ gives the following exact sequence. 
\begin{equation*}
H_2(W;\mathbb{Q})\oplus H_2(X-\textnormal{int}\,W;\mathbb{Q})\stackrel{\varphi_{\lambda}}{\to} H_2(X_{\lambda};\mathbb{Q})\stackrel{\partial_{\lambda}}{\to} H_1(\partial W;\mathbb{Q}). 
\end{equation*}
This exact sequence shows that 
\begin{equation*}
b_2(X_{\lambda})=\dim (\textnormal{Im}\, \varphi_{\lambda}) +\dim (\textnormal{Im}\, \partial_{\lambda}) \quad \text{and} \quad \dim (\textnormal{Im}\, \partial_{\lambda}) \leq b_1(\partial W).
\end{equation*}

Now we estimate the value of $G_{X_{\lambda}}(n)$ for an integer $n>b_1(\partial W)$. If $\lambda\in \Lambda$ satisfies $\dim (\textnormal{Im}\, \varphi_{\lambda})=0$, then we see that $b_2(X_{\lambda})\leq b_1(\partial W)<n$. Hence we get $G_{X_{\lambda}}(n)=0$ by the definition of the adjunction $n$-genus. Now suppose that $\lambda\in \Lambda$ satisfies $\dim (\textnormal{Im}\, \varphi_{\lambda})=k$ for some $k>0$. Then there exists a subset $\mathbf{u}_{\lambda}=\{u_1,u_2,\dots,u_k\}$ of $\mathbf{v}\cup \mathbf{w}$ such that $u_1,u_2,\dots,u_k$ are linearly independent in $H_2(X_{\lambda};\mathbb{Z})$, since $\mathbf{v}$ and $\mathbf{w}$ are rational bases of $H_2(X-\textnormal{int}\,W;\mathbb{Z})$ and $H_2(W;\mathbb{Z})$. Therefore, there exists a (possibly empty) subset $\mathbf{x}_{\lambda}=\{x_1,x_2,\dots, x_l\}$ of $H_2(X_{\lambda};\mathbb{Z})$ such that $\mathbf{u}_{\lambda}\cup \mathbf{x}_{\lambda}$ is a rational basis of $H_2(X_{\lambda};\mathbb{Z})$. We here note that 
\begin{equation*}
l=\dim (\textnormal{Im}\, \partial_{\lambda})\leq b_1(\partial W)<n.
\end{equation*}
We thus see that 
\begin{align*}
0\leq G_{X_{\lambda}}(n)&\leq G_{X_{\lambda},\mathbf{u}_{\lambda}\cup \mathbf{x}_{\lambda}}(n)\\
&\leq \max\{g_{X_{\lambda}}^{ad}(u_i)\mid1\leq i\leq k\}\\
&\leq \max\{G_{X-\textnormal{int}\,W, \mathbf{v},1}, \;G_{W,\mathbf{w},1}, \;G_{\overline{W},\mathbf{w},1}\}. 
\end{align*}
Since the last value is independent of $\lambda$, at most finitely many integers can be the values of the adjunction $n$-genera of the members of $\{X_\lambda\}_{\lambda\in \Lambda}$. 

In the case where $W$ is disconnected, each $X_\lambda$ is decomposed as $(X-\textnormal{int}\,W)\cup \widehat{W}$, where $\widehat{W}$ is a copy of $W$ possibly equipped with the reverse orientations for some of its connected components, depending on $\lambda$. Since $W$ has at most finitely many connected components due to the compactness condition of $W$, there are at most finitely many choices of $\widehat{W}$. Therefore, we can prove the desired finiteness similarly to the connected case. 
\end{proof}

\begin{proof}[Proof of Theorem~\ref{sec:twist:thm:general}]
This is straightforward from Proposition~\ref{sec:twist:prop:general}. 
\end{proof}

We next discuss surgeries.

\begin{theorem}\label{thm:genus:surgery:general}Suppose that infinitely many members of a family of compact connected oriented smooth 4-manifolds with $b_2=m$ have pairwise distinct adjunction $n$-genera for a positive integer $n$. Then, for any 4-manifold pair $(X,W)$ satisfying 
\begin{equation*}
m-b_2(X)+b_2(W)+3b_1(\partial W)<n, 
\end{equation*}
the family cannot be generated from $X$ by performing surgeries on $W$. $($Note that $W$ and $\partial W$ are not necessarily connected.$)$ 
\end{theorem}

To prove this theorem, we show the following finiteness result. 

\begin{proposition}\label{prop:genus:surgery:general}
Suppose that a family of compact connected oriented smooth 4-manifolds with $b_2=m$ is generated from $X$ by performing surgeries on $W$ for a 4-manifold pair $(X,W)$. Then, for any positive integer $n$ satisfying 
\begin{equation*}
n>m-b_2(X-\textnormal{int}\,W)+2b_1(\partial W), 
\end{equation*}
there are at most finitely many integers that can be the values of the adjunction $n$-genera of the members.
\end{proposition}

We first prove the $b_1(\partial W)=0$ case, since its proof is short and nicely demonstrates our main idea as before. 

\begin{proof}[Proof of the $b_1(\partial W)=0$ case of Proposition~\ref{prop:genus:surgery:general}]Let $\{X_\lambda\}_{\lambda\in \Lambda}$ be a family of 4-manifolds satisfying the assumption. 
Then each $X_\lambda$ has a decomposition 
\begin{equation*}
X_\lambda=(X-\textnormal{int}\,W)\cup W_\lambda
\end{equation*}
 for some compact oriented 4-manifold $W_\lambda$. We fix rational bases $\mathbf{v}=\{v_1,v_2,\dots,v_s\}$ and $\mathbf{w}_\lambda=\{w_1,w_2,\dots,w_t\}$ of $H_2(X-\textnormal{int}\,W;\mathbb{Z})$ and $H_2(W_\lambda;\mathbb{Z})$ respectively. By the assumption $b_1(\partial W)=0$, we easily see that $\mathbf{v}\cup \mathbf{w}_\lambda$ is a rational basis of $H_2(X_\lambda;\mathbb{Z})$, where we regard $\mathbf{v}$ and $\mathbf{w}_\lambda$ as subsets of $H_2(X_\lambda;\mathbb{Z})$ via the inclusions $X-\textnormal{int}\,W\hookrightarrow X_\lambda$ and $W_\lambda\hookrightarrow X_\lambda$. We can thus check that
\begin{equation*}
t=b_2(W_\lambda)=m-b_2(X-\textnormal{int}\,W)<n. 
\end{equation*}
The definition of the adjunction $n$-genus together with this inequality implies that 
\begin{equation*}
0\leq  G_{X_\lambda}(n)\leq G_{X_\lambda, \mathbf{v}\cup \mathbf{w}_\lambda,n} \leq G_{X-\textnormal{int}\,W, \mathbf{v},1}.  
\end{equation*}
Therefore, at most finitely many integers can be the values of the adjunction $n$-genera of the members of $\{X_\lambda\}_{\lambda\in \Lambda}$. 
\end{proof}

Now we prove the general case. 

\begin{proof}[Proof of Proposition~\ref{prop:genus:surgery:general}]Let $\{X_\lambda\}_{\lambda\in \Lambda}$ be a family of 4-manifolds satisfying the assumption. 
Then each $X_\lambda$ has a decomposition $X_\lambda=(X-\textnormal{int}\,W)\cup W_\lambda$ for some compact oriented 4-manifold $W_\lambda$. We fix rational bases $\mathbf{v}=\{v_1,v_2,\dots,v_s\}$ and $\mathbf{w}_\lambda=\{w_1,w_2,\dots,w_t\}$ of $H_2(X-\textnormal{int}\,W;\mathbb{Z})$ and $H_2(W_\lambda;\mathbb{Z})$ respectively. We regard each $v_i$ and $w_j$ as elements of $H_2(X_\lambda;\mathbb{Z})$ via the inclusions $X-\textnormal{int}\,W\hookrightarrow X_\lambda$ and $W_{\lambda}\hookrightarrow X_\lambda$. 

We here estimate the value of $b_2(W_{\lambda})$. The Mayer-Vietoris exact sequence for $X_\lambda=(X-\textnormal{int}\,W)\cup W_{\lambda}$ gives the following exact sequence. 
\begin{equation*}
H_2(\partial W;\mathbb{Q})\stackrel{\varphi_{\lambda}}{\to} H_2(X-\textnormal{int}\,W;\mathbb{Q})\oplus H_2(W_{\lambda};\mathbb{Q}) \stackrel{\psi_{\lambda}}{\to} H_2(X_{\lambda};\mathbb{Q})\stackrel{\partial_{\lambda}}{\to} H_1(\partial W;\mathbb{Q}). 
\end{equation*}
This exact sequence implies that 
\begin{gather*}
m=b_2(X_{\lambda})=\dim (\textnormal{Im}\, \psi_{\lambda}) + \dim (\textnormal{Im}\, \partial_{\lambda}),\\ 
\dim (\textnormal{Im}\, \varphi_{\lambda}) \: \leq \: b_1(\partial W), \\ 
 \dim (\textnormal{Im}\, \partial_{\lambda}) \: \leq \: b_1(\partial W), \\ 
b_2(X-\textnormal{int}\,W)+b_2(W_{\lambda})-b_1(\partial W) \: \leq \: \dim (\textnormal{Im}\, \psi_{\lambda}).
\end{gather*}
Using these inequalities, we can easily show that 
\begin{equation*}
b_2(W_{\lambda}) \: \leq \: m-b_2(X-\textnormal{int}\,W)+b_1(\partial W). 
\end{equation*}

Now we estimate the value of $G_{X_{\lambda}}(n)$ for a positive integer $n$ satisfying 
\begin{equation*}
n>m-b_2(X-\textnormal{int}\,W)+2b_1(\partial W). 
\end{equation*}
If $\lambda\in \Lambda$ satisfies $\dim (\textnormal{Im}\, \psi_{\lambda})<n-b_1(\partial W)$, then we easily see $b_2(X_{\lambda})<n$. Hence we get $G_{X_{\lambda}}(n)=0$ by the definition of the adjunction $n$-genus. We thus assume that $\lambda\in \Lambda$ satisfies $\dim (\textnormal{Im}\, \psi_{\lambda})=k$ for some $k\geq n-b_1(\partial W)$. 
Then the condition $\dim (\textnormal{Im}\, \psi_{\lambda})=k$ implies that there exists a subset $\mathbf{u}_{\lambda}=\{u_1,u_2,\dots,u_k\}$ of $\mathbf{v}\cup \mathbf{w}_\lambda$ such that $u_1,u_2,\dots,u_k$ are linearly independent in $H_2(X_{\lambda};\mathbb{Z})$, since $\mathbf{v}$ and $\mathbf{w}_\lambda$ are rational bases of $H_2(X-\textnormal{int}\,W;\mathbb{Z})$ and $H_2(W_{\lambda};\mathbb{Z})$. Therefore, there exists a (possibly empty) subset $\mathbf{x}_{\lambda}=\{x_1,x_2,\dots, x_l\}$ of $H_2(X_{\lambda};\mathbb{Z})$ such that $\mathbf{u}_{\lambda}\cup \mathbf{x}_{\lambda}$ is a rational basis of $H_2(X_{\lambda};\mathbb{Z})$. Note that $l=\dim (\textnormal{Im}\, \partial_{\lambda})\leq b_1(\partial W)$, since $b_2(X_{\lambda})=k + \dim (\textnormal{Im}\, \partial_{\lambda})$. Due to the above estimate of $b_2(W_{\lambda})$ and the assumption on $n$, we thus obtain 
\begin{equation*}
b_2(W_{\lambda})+l<n. 
\end{equation*}
This inequality shows that the number of elements of $(\mathbf{u}_{\lambda}\cap \mathbf{w}_{\lambda})\cup \mathbf{x}_{\lambda}$ is less than $n$. 
We thus obtain the following estimate of $G_{X_{\lambda}}(n)$. 
\begin{align*}
0\leq G_{X_{\lambda}}(n)&\leq G_{X_{\lambda},\mathbf{u}_{\lambda}\cup \mathbf{x}_{\lambda}}(n)\\
&\leq \max\{g_{X_{\lambda}}^{ad}(u_i)\mid u_i\in \mathbf{v}\}\\
&\leq G_{X-\textnormal{int}\,W, \mathbf{v},1}. 
\end{align*}
Therefore, at most finitely many integers can be the values of the adjunction $n$-genera of the members of $\{X_\lambda\}_{\lambda\in \Lambda}$. 
\end{proof}

\begin{proof}[Proof of Theorem~\ref{thm:genus:surgery:general}]For a 4-manifold pair $(X, W)$, using the Mayer-Vietoris exact sequence for the decomposition $X=(X-\textnormal{int}\,W)\cup W$, we can easily show that 
\begin{equation*}
b_2(W)+b_2(X-\textnormal{int}\,W)\: \geq \: b_2(X)-b_1(\partial W),
\end{equation*}
similarly to the estimate of $b_2(W_{\lambda})$ in the proof of Proposition~\ref{prop:genus:surgery:general}. Hence if a positive integer $n$ satisfies 
\begin{equation*}
n>m-b_2(X)+b_2(W)+3b_1(\partial W),
\end{equation*}
then the above inequality shows
\begin{equation*}
n>m-b_2(X-\textnormal{int}\,W)+2b_1(\partial W). 
\end{equation*}
Theorem~\ref{thm:genus:surgery:general} thus follows from Proposition~\ref{prop:genus:surgery:general}. 
\end{proof}

Finally, we discuss twists without fixing embeddings. Applying Theorem~\ref{thm:genus:surgery:general}, we prove the following results. 

\begin{theorem}\label{thm:genus:twist:embedding}Suppose that infinitely many members of a family of compact connected oriented smooth 4-manifolds with $b_2=m$ have pairwise distinct adjunction $n$-genera for a positive integer $n$. Then, for any compact oriented smooth 4-manifolds $X$ and $W$ satisfying 
\begin{equation*}
m-n<b_2(W)-4b_1(\partial W), 
\end{equation*}
the family cannot be generated from $X$ by twisting an embedded copy of $W$ and varying the embedding of $W$ into $X$. $($Note that $W$ and $\partial W$ are not necessarily connected.$)$ 
\end{theorem}
\begin{proof}
Assume that an infinite family of 4-manifolds satisfies the assumption of this theorem. Suppose, to the contrary, that there exist compact oriented smooth 4-manifolds $X$ and $W$ with 
\begin{equation*}
m-n<b_2(W)-4b_1(\partial W)
\end{equation*}
such that the family is generated from $X$ by twisting an embedded copy of $W$ and varying the embedding of $W$ into $X$. Let $W_0$ be an embedded copy of $W$ in $X$. Then the above assumption on $X$ and $W$ guarantees that the family is generated from $X$ by performing surgeries on $X-\textnormal{int}\, W_0$. We can easily check that 
\begin{equation*}
b_2(X-\textnormal{int}\, W_0) \: \leq \: b_2(X)-b_2(W)+b_1(\partial W), 
\end{equation*}
similarly to the estimate of $b_2(W_{\lambda})$ in the proof of Proposition~\ref{prop:genus:surgery:general}. Hence we see that 
\begin{equation*}
n>m-b_2(X)+b_2(X-\textnormal{int}\, W_0)+3b_1(\partial W). 
\end{equation*}
Since infinitely many members of the family have pairwise distinct adjunction $n$-genera, this inequality contradicts Theorem~\ref{thm:genus:surgery:general}. 
\end{proof}

\begin{remark}As seen from the proofs, all the results in this section still hold without the compactness condition of the members of a family of 4-manifolds, if the second homology group of each member is finitely generated. 
\end{remark}

\section{Estimating adjunction $n$-genera}\label{sec:estimate}
In this section, we give sufficient conditions that infinitely many members of a family of 4-manifolds to have pairwise distinct adjunction $n$-genera, by utilizing the adjunction inequalities. These conditions also provide simple but effective methods for coarsely distinguishing smooth structures. 

Let $\{X_{i}\}_{i\in \mathbb{N}}$ be a family of compact connected oriented smooth 4-manifolds, and let $S_{i}$ $(i\in \mathbb{N})$ be a finite subset of $H^2(X_{i};\mathbb{Z})$. We introduce the following condition, which plays an important role for estimating the adjunction $n$-genera. 

\begin{definition}\label{def:n-strongly inequivalent}For a positive integer $n$, we say that the members of $\{S_{i}\}_{i\in \mathbb{N}}$ are pairwise \textit{$n$-nicely inequivalent}, 
if for each $i\in \mathbb{N}$ there exist a positive integer $m_i\geq n$, non-zero classes $K_1^{(i)}, K_2^{(i)}, \dots, K_{m_i}^{(i)}$ of $H^2(X_{i};\mathbb{Z})$, and non-negative integers $a^{(i)}_{1},a^{(i)}_{2},\dots, a^{(i)}_{m_i}$ satisfying the following three conditions. 
\begin{itemize} 
 \item $S_i=\{\pm a^{(i)}_{1}K_1^{(i)}\pm a^{(i)}_{2}K_2^{(i)}\pm \dots \pm a^{(i)}_{m_i}K_{m_i}^{(i)}\}$ for each $i\in \mathbb{N}$. 
 \item $K_1^{(i)}, K_2^{(i)}, \dots, K_n^{(i)}$ are primitive and linearly independent for each $i\in \mathbb{N}$. 
 \item For each $1\leq j\leq n$, the integer sequence $\{a_j^{(i)}\}_{i\in \mathbb{N}}$ is strictly increasing. 
\end{itemize}
\end{definition}

We first discuss the case where the members are closed 4-manifolds. For a closed connected oriented smooth 4-manifold $X$ with $b_2^+>1$, let $\mathcal{B}_{X}$ denote the set of Seiberg-Witten basic classes of $X$. 

\begin{theorem}\label{thm:estimate:SW}Let $\{X_{i}\}_{i\in \mathbb{N}}$ be an infinite family of connected closed oriented smooth 4-manifolds with $b_2^+>1$ that are of Seiberg-Witten simple types. Suppose each $\mathcal{B}_{X_i}$ has a subset $S_i$ such that the members of $\{S_{i}\}_{i\in \mathbb{N}}$ are pairwise $n$-nicely inequivalent for a positive integer $n$. Then at least infinitely many members have pairwise distinct adjunction $n$-genera. 
\end{theorem}
\begin{remark}Many families produced by logarithmic transforms or Fintushel-Stern's knot surgeries satisfy this assumption, as we observe using the formulas of Seiberg-Witten invariants (cf.\ \cite{FS4}). We note that, if a closed oriented smooth 4-manifold with $b_2^+>1$ contains a smoothly embedded torus with the self-intersection number $0$ representing a non-torsion second homology class, then the 4-manifold is of Seiberg-Witten simple type, according to \cite{MMS}. 
\end{remark}

\begin{proof} By the assumption, for each $i\in \mathbb{N}$ there exist a positive integer $m_i\geq n$, non-zero classes $K_1^{(i)}, K_2^{(i)}, \dots, K_{m_i}^{(i)}$ of $H^2(X_{i};\mathbb{Z})$, and non-negative integers $a^{(i)}_{1},a^{(i)}_{2},\dots, a^{(i)}_{m_i}$ satisfying the following conditions. 
\begin{itemize} 
 \item [(i)] $S_i=\{\pm a^{(i)}_{1}K_1^{(i)}\pm a^{(i)}_{2}K_2^{(i)}\pm \dots \pm a^{(i)}_{m_i}K_{m_i}^{(i)}\}$ for each $i\in \mathbb{N}$. 
 \item [(ii)] $K_1^{(i)}, K_2^{(i)}, \dots, K_n^{(i)}$ are primitive and linearly independent for each $i\in \mathbb{N}$. 
 \item [(iii)] For each $1\leq j\leq n$, the integer sequence $\{a_j^{(i)}\}_{i\in \mathbb{N}}$ is strictly increasing.
\end{itemize}
Let $\mathbf{v}_i=\{v_1,\dots,v_k\}$ be a rational basis of $H_2(X_{i};\mathbb{Z})$. Since $K_1^{(i)}, K_2^{(i)}, \dots, K_n^{(i)}$ are linearly independent, the universal coefficient theorem implies that there exist pairwise distinct elements $v_{t_1},v_{t_2},\dots, v_{t_n}$ of $\mathbf{v}_i$ such that $\langle K_j^{(i)}, v_{t_j}\rangle \neq 0$ for each $1\leq j\leq n$. 
We note that each element of $S_i$ is a Seiberg-Witten basic class of $X_i$. Due to the condition (i), the adjunction inequality to each $v_{t_j}$ shows that 
\begin{equation*}
\sum_{l=1}^{m_i} \left| \langle a_l^{(i)}K_l^{(i)}, \; v_{t_j}\rangle\right|\; +\; v_{t_j}\cdot v_{t_j}\; \leq \; \max\{0, 2g_{X_{i}}(v_{t_j})-2\}. 
\end{equation*}
This implies that $a_j^{(i)}\leq \; g^{ad}_{X_{i}}(v_{t_j})$ for each $1\leq j\leq n$. It thus follows that 
\begin{equation*}
\min\{a_1^{(i)}, a_2^{(i)}, \dots, a_n^{(i)}\}\leq \; \min\{g^{ad}_{X_{i}}(v_{t_1}), g^{ad}_{X_{i}}(v_{t_2}),\dots, g^{ad}_{X_{i}}(v_{t_n})\}\leq G_{X_i, \mathbf{v}, n}. 
\end{equation*}
Hence we obtain 
\begin{equation*}
\min\{a_1^{(i)}, a_2^{(i)}, \dots, a_n^{(i)}\} \leq \; G_{X_i}(n). 
\end{equation*}
The condition (iii) thus shows $\lim_{i\to \infty} G_{X_{i}}(n)=\infty$. Therefore, at least 
infinitely many members of $\{X_{i}\}_{i\in \mathbb{N}}$ have pairwise distinct adjunction $n$-genera. 
\end{proof}

We next discuss the case where the members admit Stein structures. For a compact connected oriented smooth 4-manifold $X$ with boundary, we define a subset $\mathcal{C}_S(X)$ of $H^2(X;\mathbb{Z})$ by 
\begin{equation*}
\mathcal{C}_S(X)=\{\pm c_1(X,J)\in H^2(X;\mathbb{Z})\mid \text{$J$ is a Stein structure on $X$}\}, 
\end{equation*}
where $c_1(X,J)$ denotes the first Chern class of $X$ equipped with a Stein structure $J$. 

\begin{theorem}\label{thm:estimate:Stein}Let $\{X_{i}\}_{i\in \mathbb{N}}$ be an infinite family of compact connected oriented smooth 4-manifolds with boundary. 
Suppose each $\mathcal{C}_S(X_i)$ has a subset $S_i$ such that the members of $\{S_{i}\}_{i\in \mathbb{N}}$ are pairwise $n$-nicely inequivalent for a positive integer $n$. Then at least infinitely many members have pairwise distinct adjunction $n$-genera. 
\end{theorem}
\begin{proof}Since a version of the adjunction inequality holds for any compact Stein 4-manifold (\cite{LM1}, \cite{AM}, \cite{LM2}), the above proof of Theorem~\ref{thm:estimate:SW} works in this case as well. 
\end{proof}

This sufficient condition is of independent interest, since it is a simple sufficient condition that an infinite family has infinitely many pairwise non-diffeomorphic members. Note that the adjunction $n$-genus is a diffeomorphism invariant.  

We here recall that the \textit{divisibility} $d(\alpha)$ of an element $\alpha$ in a finitely generated abelian group $G$ is defined by 
\begin{equation*}
d(\alpha)=\left\{
\begin{array}{ll}
\max\{n\in \mathbb{Z}\mid \text{$\alpha=n\alpha'$ for some $\alpha'\in G$}\},&\text{if $\alpha$ is not torsion;}\\
0, &\text{if $\alpha$ is torsion.}
\end{array}
\right.
\end{equation*}
We immediately obtain the following corollary. 

\begin{corollary}\label{cor:estimate:Stein}For a given infinite family of compact Stein 4-manifolds, if the divisibilities of the first Chern classes of the members are pairwise distinct, then at least infinitely many members have pairwise distinct adjunction $1$-genera. 
\end{corollary}

\section{Examples}\label{sec:example}
In this section, we prove our main results. To do this, we give exotic 4-manifolds with distinct adjunction $n$-genera by applying results in Section~\ref{sec:estimate}. 

\begin{theorem}\label{sec:example:thm:example}For each positive integer $n$, there exists an infinite family of homeomorphic but pairwise non-diffeomorphic simply connected closed oriented smooth 4-manifolds with pairwise distinct adjunction $n$-genera. Furthermore, there exist infinitely many such families whose homeomorphism types are pairwise distinct. 
\end{theorem}

\begin{theorem}\label{sec:example:thm:example2}For each positive integer $n$, there exists an infinite family of homeomorphic but pairwise non-diffeomorphic simply connected compact oriented Stein 4-manifolds with pairwise distinct adjunction $n$-genera. Furthermore, there exist infinitely many such families whose homeomorphism types are pairwise distinct. 
\end{theorem}

The former theorem gives examples of closed 4-manifolds, and the latter gives examples of compact 4-manifolds having boundary. By results in Section~\ref{sec:nonexistence},  Theorems~\ref{intro:thm:twists} and \ref{intro:thm:surgery} follow from these examples. 

\begin{proof}[Proof of Theorems~\ref{intro:thm:twists} and \ref{intro:thm:surgery}] These follow from Theorems~\ref{sec:example:thm:example}, \ref{sec:twist:thm:general} and \ref{thm:genus:surgery:general}. 
\end{proof}

We give examples of Theorem~\ref{sec:example:thm:example} and \ref{sec:example:thm:example2} and then prove Theorems~\ref{intro:thm:embedding:closed} and \ref{intro:thm:embedding:boundary}. In the rest, we use the following notations. For a positive integer $n$, let $E(n)$ be the simply connected elliptic surface with Euler characteristic $12n$ and with no multiple fibers, and let $E(n)_{p,q}$ be the elliptic surface obtained from $E(n)$ by performing logarithmic transformations of multiplicities $p$ and $q$ along two tubular fibers. We denote the Gompf nucleus \cite{G0} of $E(n)$ by $N(n)$, namely, $N(n)$ is the tubular neighborhood of the union of a cusp fiber and a $-n$-section of $E(n)$. 

\subsection{The $n=1$ case of Theorems~\ref{sec:example:thm:example} and \ref{sec:example:thm:example2}}\label{subsec:ex:n=1} We start with the simplest case of Theorems~\ref{sec:example:thm:example} and \ref{sec:example:thm:example2}. We note that the following provides a justification for Example~\ref{intro:ex}. 

\subsubsection{Elliptic surfaces} Any given infinite family of homeomorphic but pairwise non-diffeomorphic simply connected closed elliptic surfaces with $b_2^+>1$ has infinitely members with pairwise distinct adjunction $1$-genera. This can be seen as follows. As is well-known (cf.\ \cite{GS}), any simply connected elliptic surface with $b_2^+>1$ is diffeomorphic to $E(k)_{p,q}$ for some positive integers $k\geq 2$ and $p,q\geq 1$. Since the members of the infinite family are pairwise homeomorphic, the value of $k$ is independent of a member. According to the computation of $SW_{E(k)_{p,q}}$ in \cite{FS1} (see Theorem~3.3.6 in \cite{GS}), each $E(k)_{p,q}$ has a Seiberg-Witten basic class whose divisibility is $kpq-p-q$. Therefore, Theorem~\ref{thm:estimate:SW} shows that infinitely many members have pairwise distinct adjunction $1$-genera. 
\subsubsection{Knot surgery} 
Suppose that a simply connected closed oriented smooth 4-manifold $X$ with $b_2^+>1$ contains a homologically non-trivial torus $T$ of self-intersection $0$ satisfying $\pi_1(X-T)\cong 1$ and has a non-trivial Seiberg-Witten invariant. There are many examples of $X$ such as elliptic surfaces. Let $X_K$ denote the 4-manifold obtained from $X$ by performing a Fintushel-Stern's knot surgery~\cite{FS2} on $T$ using a knot $K$ in $S^3$. Then for any infinite family $\{K_i\}_{i\in \mathbb{N}}$ of knots in $S^3$ whose Alexander polynomials have pairwise distinct degrees, the members of the infinite family $\{X_{K_i}\}_{i\in \mathbb{N}}$ are homeomorphic but pairwise non-diffeomorphic simply connected closed 4-manifolds due to \cite{FS2} and \cite{Su}. We here recall that $K\in \mathcal{B}_X$, if and only if $-K\in \mathcal{B}_X$ (cf.\ Remark~2.4.4.(d) in \cite{GS}). Thus, using the Fintushel-Stern's knot surgery formula (\cite{FS2}. cf.\ \cite{F_knot}), one can check that each $\mathcal{B}_{X_{K_i}}$ has a subset $S_i$ such that the members of $\{S_i\}_{i\in \mathbb{N}}$ are pairwise $1$-nicely inequivalent. Therefore, by Theorem~\ref{thm:estimate:SW}, an infinite subfamily of $\{X_{K_i}\}_{i\in \mathbb{N}}$ is an example of the $n=1$ case of Theorem~\ref{sec:example:thm:example}. 
\subsubsection{Small Stein 4-manifolds}\label{subsec:twist:Stein} The infinite families of homeomorphic but pairwise non-diffeomorphic simply connected compact Stein 4-manifolds with $b_2=2$ obtained in \cite{AY6} and \cite{Y6} have infinite subfamilies that are examples of the $n=1$ case of Theorem~\ref{sec:example:thm:example2}. To prove this claim, we use Stein structures obtained in \cite{AY7}, which are different from the original ones obtained in \cite{AY6} and \cite{Y6}. (The original Stein structures works as well, but those in \cite{AY7} are the simplest regarding our purpose). 

Let $m=(m_0,m_1,m_2)$ be a 3-tuple of integers satisfying $m_0\geq 2$, $m_1\geq 2$, $m_2\geq 1$, and let $X^{(m)}$ and $X^{(m)}_p$ be the 4-manifolds in Figure~\ref{fig:smooth_nuclei}, where $p$ is a positive integer. (These manifolds were introduced in \cite{AY6} and \cite{Y6}, and this figure was used in \cite{AY7}.) 
Then, for each fixed $m$, infinitely many members of $\{X_{2p}^{(m)}\}_{p\in \mathbb{N}}$ are homeomorphic but pairwise non-diffeomorphic Stein 4-manifolds with $b_2=2$ (\cite{AY6}, \cite{Y6}), where the Stein structure of each $X_p^{(m)}$ is the one given by the Stein handlebody diagram in Figure~\ref{fig:Stein_nucleus}. This diagram was obtained in \cite{AY7}, and the boxes are the Legendrian versions of left-handed ful twists shown in Figure~\ref{fig:Legendrian_left_twists}. By the formula of the first Chern classes of Stein 4-manifolds in \cite{G1} together with Lemma~4.1 in \cite{AY7}, we can easily check that the divisibility of $c_1(X_p^{(m)})$ is $p(m_1-1)+m_0-2$.  Hence Corollary~\ref{cor:estimate:Stein} shows that an infinite subfamily of $\{X_{2p}^{(m)}\}_{p\in \mathbb{N}}$ is an example of the $n=1$ case of Theorem~\ref{sec:example:thm:example2}. 

\begin{figure}[h!]
\begin{center}
\includegraphics[width=4.1in]{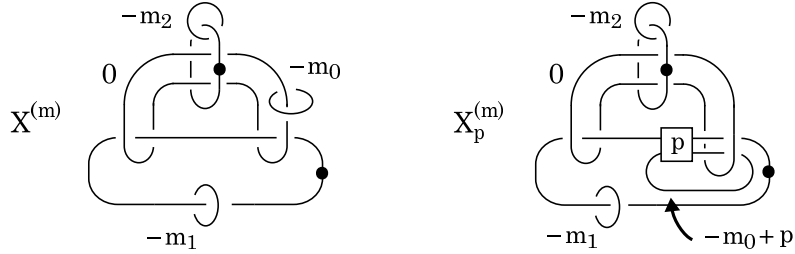}
\caption{$X^{(m)}$ and $X^{(m)}_p$ ($m_0\geq 2$,\, $m_1\geq 2$,\, $m_2\geq 1$,\, $p\geq 1$).}
\label{fig:smooth_nuclei}
\end{center}
\end{figure}

\begin{figure}[h!]
\begin{center}
\includegraphics[width=3.7in]{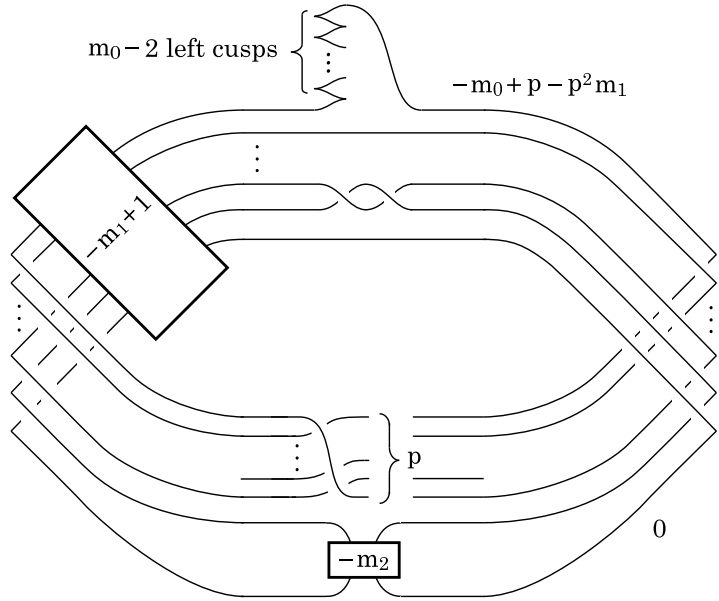}
\caption{A Stein handlebody diagram of $X_p^{(m)}$}
\label{fig:Stein_nucleus}
\end{center}
\end{figure}

\begin{figure}[h!]
\begin{center}
\includegraphics[width=3.7in]{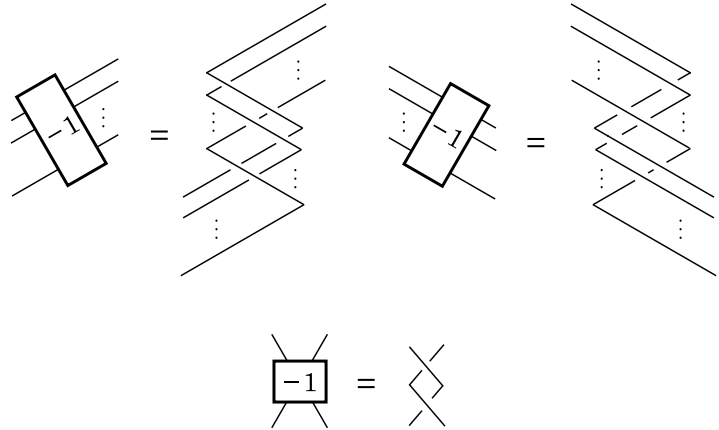}
\caption{Legendrian versions of left-handed full twists}
\label{fig:Legendrian_left_twists}
\end{center}
\end{figure}
\subsection{The general case of Theorems~\ref{sec:example:thm:example} and \ref{sec:example:thm:example2}}\label{subsec:ex:general}
Next we give general examples of Theorems~\ref{sec:example:thm:example} and \ref{sec:example:thm:example2} and then prove Theorems~\ref{intro:thm:embedding:closed} and \ref{intro:thm:embedding:boundary}. In the rest, let $n$ be a positive integer. 
\subsubsection{Knot surgery}\label{sec:ex:subsec:general_knot}
 Suppose that $X$ is a simply connected closed oriented smooth 4-manifold with $b_2^+>1$ that has a non-trivial Seiberg-Witten invariant and contains pairwise disjoint $n$ copies of the Gompf nucleus $N(2)$. For example, the elliptic surface $E(k)$ $(k\geq 2)$ contains pairwise disjoint $2(k-1)$ copies of the Gompf nucleus $N(2)$ (see Exercises 3.1.12.(c) in \cite{GS}). One can also construct many examples of $X$, since the boundary connected sum $\natural_{n} N(2)$ admits a Stein structure, and any compact Stein 4-manifold can be embedded into a simply connected closed minimal symplectic 4-manifold with $b_2^+>1$ (\cite{AO3}). 
 
For each $1\leq i\leq n$, let $T_i$ be a regular fiber of a cusp neighborhood of the $i$-th copy of $N(2)$ embedded in $X$. Then each $T_i$ is a torus with the self-intersection number $0$, and $T_1,T_2,\dots, T_n$ represent linearly independent second homology classes of $X$. For an infinite family $\{K_i\}_{i\in \mathbb{N}}$ of knots in $S^3$ whose Alexander polynomials have pairwise distinct degrees, let $X_i$ be the 4-manifold obtained from $X$ by performing knot surgeries on the tori $T_1,\dots, T_n$ using the knot $K_i$. Then each $X_i$ is homeomorphic to $X$. 
 On the other hand, due to the knot surgery formula (\cite{FS2}), one can see that each $\mathcal{B}_{X_i}$ has a subset $S_i$ such that the members of $\{S_{i}\}_{i\in \mathbb{N}}$ are pairwise $n$-nicely inequivalent. Therefore,  Theorem~\ref{thm:estimate:SW} shows that an infinite subfamily of $\{X_i\}_{i\in \mathbb{N}}$ is an example of Theorem~\ref{sec:example:thm:example}. We note that $b_2(X_i)=12n+10$ in the case where $X=E(n+1)$. 
 
\subsubsection{Stein 4-manifolds}\label{sec:ex:subsec:general_Stein} We use the notation in Section~\ref{subsec:twist:Stein}, and we fix $m$. We here observe that $X_p^{(m)}$ admits (at least) two Stein structures: one is given by the Stein handlebody diagram in Figure~\ref{fig:Stein_nucleus}, and the other is given by the diagram in Figure~\ref{fig:Stein_nucleus_right}. We easily see that the first Chern classes of these two Stein structures are $r(p,m)K^p$ and $-r(p,m)K^p$ for some primitive second cohomology class $K^p$ of $X_p^{(m)}$, where $r(p,m):=p(m_1-1)+m_0-2$. Now let $Z_{n,2p}$ be the boundary connected sum $\natural_n X_{2p}^{(m)}$. Then the members of $\{Z_{n,2p}\}_{p\in \mathbb{N}}$ are pairwise homeomorphic. By the above observation, we see that $\mathcal{C}_S(X_p^{(m)})$ has the subset $S_p:=\{r(p,m)(\pm K_1^p\pm K_2^p\pm \dots \pm K_n^p)\}$, where each $K_i^p$ is the class of $H^2(Z_{n,2p};\mathbb{Z})\cong \oplus_{n}H^2(X_{2p}^{(m)};\mathbb{Z})$ corresponding to $K^p$ in the $i$-th copy of $H^2(X_{2p}^{(m)};\mathbb{Z})$. Hence the members of $\{S_{p}\}_{p\in \mathbb{N}}$ are pairwise $n$-nicely inequivalent. Thus by Theorem~\ref{thm:estimate:Stein}, we see that an infinite subfamily of $\{Z_{n,2p}\}_{p\in \mathbb{N}}$ is an example of Theorem~\ref{sec:example:thm:example2}. We note that $b_2(Z_{n,2p})=2n$. 

\begin{figure}[h!]
\begin{center}
\includegraphics[width=3.7in]{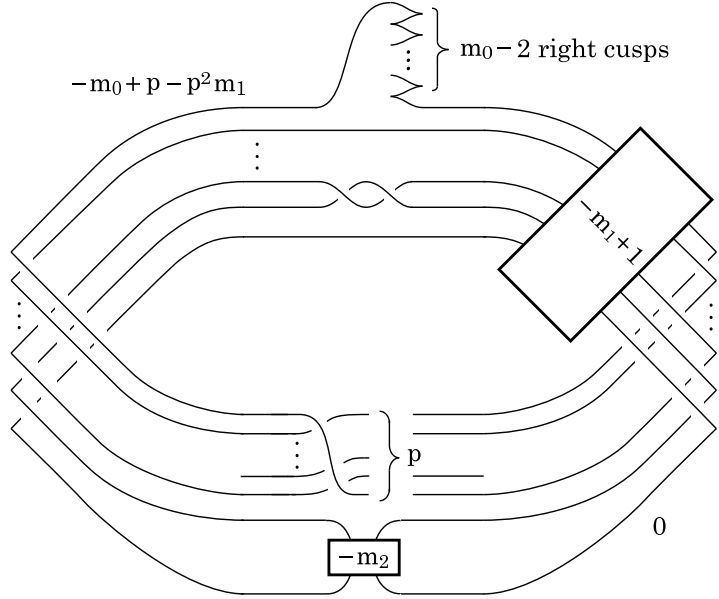}
\caption{Another Stein handlebody diagram of $X_p^{(m)}$}
\label{fig:Stein_nucleus_right}
\end{center}
\end{figure}

\subsection{Proofs of Theorems~\ref{intro:thm:embedding:closed} and \ref{intro:thm:embedding:boundary}} Lastly we prove Theorems~\ref{intro:thm:embedding:closed} and \ref{intro:thm:embedding:boundary} by using the above examples. These theorems immediately follow from the examples in Subsections~\ref{sec:ex:subsec:general_knot} and \ref{sec:ex:subsec:general_Stein} and Theorem~\ref{thm:genus:twist:embedding}. 
\subsection*{Acknowledgements}The author would like to thank Selman Akbulut and Motoo Tange for useful discussions and helpful comments. The author would also like to thank the referee for useful comments. The author was partially supported by JSPS KAKENHI Grant Numbers 16K17593, 26287013 and 17K05220. 


\end{document}